\documentclass[12pt,a4paper,english,reqno]{amsart}
\usepackage[a4paper,footskip=1.5em]{geometry}
\usepackage{amsmath,amssymb,amsthm,mathtools}
\usepackage[mathscr]{euscript}
\usepackage[usenames,dvipsnames]{color}
\usepackage{adjustbox,tikz,calc,graphics,babel,standalone}
\usepackage{subcaption}
\usepackage{csquotes,enumerate,verbatim}
\usepackage[final]{microtype}
\usepackage[numbers]{natbib}
\usetikzlibrary{shapes.misc,calc,intersections,patterns,decorations.pathreplacing}
\usepackage{hyperref}
\hypersetup{colorlinks=true,linkcolor=blue,citecolor=blue,pdfpagemode=UseNone,pdfstartview={XYZ null null 1.00}}
\usepackage{cmtiup}

\theoremstyle{plain}
\newtheorem*{theorem*}{Theorem}
\newtheorem{theorem}{Theorem}[section]
\newtheorem{lemma}[theorem]{Lemma}
\newtheorem{claim}[theorem]{Claim}
\newtheorem{proposition}[theorem]{Proposition}
\newtheorem*{claim*}{Claim}

\newtheorem{conjecture}[theorem]{Conjecture}

\theoremstyle{remark}

\def\N{\mathbb{N}}
\def\Z{\mathbb{Z}}

\def\C{\mathcal}

\DeclareMathOperator\Deg{d}
\DeclareMathOperator\wdeg{\zeta}

\DeclareMathOperator{\EV}{\mathbb{E}}
\DeclareMathOperator{\PV}{\mathbb{P}}

\let\eps\varepsilon

\let\originalleft\left
\let\originalright\right
\renewcommand{\left}{\mathopen{}\mathclose\bgroup\originalleft}
\renewcommand{\right}{\aftergroup\egroup\originalright}

\makeatletter
\def\imod#1{\allowbreak\mkern10mu({\operator@font mod}\,\,#1)}
\makeatother

\begin{document}
	
	\title{Ramsey graphs induce subgraphs of many different sizes}
	
	\author{Bhargav Narayanan}
	\address{Department of Pure Mathematics and Mathematical Statistics, University of Cambridge, Wilberforce Road, Cambridge CB3\thinspace0WB, UK}
	\email{b.p.narayanan@dpmms.cam.ac.uk}
	
	\author{Julian Sahasrabudhe}
	\address{Department of Mathematical Sciences, University of Memphis, Memphis TN 38152, USA}
	\email{julian.sahasra@gmail.com}
	
	\author{Istv\'{a}n Tomon}
	\address{Department of Pure Mathematics and Mathematical Statistics, University of Cambridge, Wilberforce Road, Cambridge CB3\thinspace0WB, UK}
	\email{i.tomon@dpmms.cam.ac.uk}
	
	\date{28 March 2016}
	
	\begin{abstract}
		A graph on $n$ vertices is said to be \emph{$C$-Ramsey} if every clique or independent set of the graph has size at most $C \log n$. The only known constructions of Ramsey graphs are probabilistic in nature, and it is generally believed that such graphs possess many of the same properties as dense random graphs. Here, we demonstrate one such property: for any fixed $C>0$, every $C$-Ramsey graph on $n$ vertices induces subgraphs of at least $n^{2-o(1)}$ distinct sizes. This near-optimal result is closely related to two unresolved conjectures, the first due to Erd\H{o}s and McKay and the second due to Erd\H{o}s, Faudree and S\'{o}s, both from 1992.
	\end{abstract}
	
	\maketitle
	
	\section{Introduction}
	A subset of the vertices of a graph is called \emph{homogeneous} if it induces either a clique or an independent set. Graphs with no large homogeneous sets are central objects in graph Ramsey theory and the properties of such graphs have been investigated by many researchers over the last sixty years. Erd\H{o}s and Szekeres~\citep{lowerramsey} proved a quantitative form of Ramsey's foundational result~\citep{Ramsey1930} and showed that every graph on $n$ vertices contains a homogeneous set of size at least $(\log n)/2$ and subsequently, Erd\H{o}s~\citep{upperramsey} used probabilistic techniques to show the existence of an $n$-vertex graph with no homogeneous sets of size greater than $2\log n$; here, and throughout the paper, all logarithms are base $2$. It is generally believed that graphs containing no large homogeneous sets should resemble random graphs; indeed, it is worth noting that in spite of considerable effort, see~\citep{fw, bip_const} for example, all known deterministic constructions of graphs with small homogeneous sets are substantially weaker than the original construction of Erd\H{o}s.
	
	Let $\hom(G)$ denote the size of the largest homogeneous set of a graph $G$. For a positive constant $C>0$, we say that a graph $G$ on $n$ vertices is \emph{$C$-Ramsey} if $\hom(G) \le C\log n$. In addition to the lack of deterministic constructions, the intuition that Ramsey graphs are `random-like' is also supported by rigorous results, proved over the course of the last forty years, that show that such graphs share various properties with dense random graphs. One of the first such results is due to Erd\H{o}s and Szemer\'{e}di~\citep{dense} who proved that the edge density of Ramsey graphs must be bounded away from both $0$ and $1$. Examples of more recent results include a theorem of Shelah~\citep{shelah} that asserts that every $C$-Ramsey graph on $n$ vertices contains $2^{\delta n}$ non-isomorphic induced subgraphs, and a theorem due to Pr\"{o}mel and R\"{o}dl~\citep{universal} that asserts that every $n$-vertex $C$-Ramsey graph contains an induced copy of all graphs on at most $\delta \log n $ vertices; in both cases, $\delta > 0$ is a constant that depends only on $C$.
	
	In this paper, we investigate the set of sizes of induced subgraphs of a Ramsey graph. For a graph $G$, writing $v(H)$ and $e(H)$ respectively for the number of vertices and edges of a graph $H$, let
	\[ \Phi(G)=\{e(H): H \text{ is an induced subgraph of }G \}\]
	and let
	\[ \Psi(G)=\{(v(H), e(H)): H \text{ is an induced subgraph of }G \}.\]
	If $G$ is a $C$-Ramsey graph on $n$ vertices, then it is conjecturally believed that the sets $\Phi(G)$ and $\Psi(G)$ behave like the sets $\Phi(\tilde G)$ and $\Psi(\tilde G)$, where $\tilde G \sim G(n,p)$ is a typical dense random graph on $n$ vertices with an appropriately chosen edge density $p = p(C)>0$. We describe two conjectures that make this idea precise below.
	
	First, Erd\H{o}s and McKay~\citep{erdos1,erdos2} conjectured that if $G$ is a Ramsey graph, then $\Phi(G)$ must be large in the following strong sense.
	\begin{conjecture}\label{strongconj}
		For any $C>0$, there exists a $\delta = \delta(C) > 0$ such that
		\[ \{0, 1, \dots, \delta n^2 \} \subset \Phi(G)\]
		for every $C$-Ramsey graph  $G$ on $n$ vertices.
	\end{conjecture}
	
	Next, Erd\H{o}s, Faudree and S\'{o}s~\citep{erdos1,erdos2} made a similar conjecture that if $G$ is a Ramsey graph, then $\Psi(G)$ must also be large.
	\begin{conjecture}\label{weakconj}
		For any $C>0$, there exists a $\delta = \delta(C) > 0$ such that \[|\Psi(G)| \ge \delta n^{5/2}\] for every $C$-Ramsey graph  $G$ on $n$ vertices.
	\end{conjecture}
	
	Towards Conjecture~\ref{weakconj}, Alon and Kostochka~\citep{pairs2} proved that $|\Psi(G)| = \Omega(n^{2})$ for every $n$-vertex Ramsey graph $G$. This result was subsequently improved by Alon, Balogh, Kostochka and Samotij~\citep{pairs1} who showed that if $G$ is a Ramsey graph on $n$ vertices, then $|\Psi(G)| = \Omega(n^{2.369})$.
	
	In comparison to Conjecture~\ref{weakconj}, our understanding of Conjecture~\ref{strongconj} is quite poor. The best result in the direction of this conjecture is due to Alon, Krivelevich and Sudakov~\citep{interval} who proved that for any $C>0$, there exists a $\delta = \delta(C) > 0$ such that if $G$ is a $C$-Ramsey graph on $n$ vertices, then $\{0,1, \dots, n^{\delta}\} \subset \Phi(G)$.
	
	In the light of Conjecture~\ref{strongconj}, one would expect that $|\Phi(G)|=\Omega(n^{2})$ for every $n$-vertex Ramsey graph $G$. However, even this weakening of Conjecture~\ref{strongconj} is not known. Indeed, while it follows immediately from the aforementioned result of Alon, Balogh, Kostochka and Samotij that $|\Phi(G)| = \Omega(n^{1.369})$ for any $n$-vertex Ramsey graph $G$, we are unaware of any other nontrivial bounds. In this paper, we fill this void with the following near-optimal result.
	\begin{theorem}\label{mainthm}
		Let $C,\eps>0$ be positive real numbers. If $n\in\N$ is sufficiently large, then \[|\Phi(G)| \ge n^{2-\eps}\]
		for every $C$-Ramsey graph $G$ on $n$ vertices.
	\end{theorem}
	
	The main ingredient in the proof of this theorem is a lower bound on the cardinality of (an appropriate generalisation of) the set $\Psi$ for `vertex-weighted' Ramsey graphs; this might be of independent interest.
	
	This paper is organised as follows. In Section~\ref{prelim}, we set out some notation and collect together a few useful facts. We give an overview of our approach in Section~\ref{sketch}, and the proof of Theorem~\ref{mainthm} proper in Section~\ref{proof}. We discuss some problems and conclude this note in Section~\ref{conc}. For the sake of clarity of presentation, we systematically omit floor and ceiling signs whenever they are not crucial.
	
	\section{Preliminaries}\label{prelim}
	In this short section, we introduce some notation and collect together some facts that we shall make use of in the sequel.
	
	Our notation is mostly standard. As usual, given a graph $G=(V,E)$, we write $v(G)$ and $e(G)$ respectively for the number of vertices and edges of $G$, and we call the ratio
	\[e(G)\binom{v(G)}{2}^{-1}\]
	the \emph{edge density of $G$}. For $U\subset V$, we denote by $G[U]$ the subgraph of $G$ induced by $U$. For a vertex $x \in V$, we denote its neighbourhood by $\Gamma(x)$ and write $\Deg(x)=|\Gamma(x)|$ for its degree; also, for $U \subset V$, let $\Gamma_{U}(x)=\Gamma(x)\cap U$ and $\Deg_{U}(x)=|\Gamma_{U}(x)|$.
	
	A quick disclaimer about the word `size' is perhaps in order. Since this paper is concerned primarily with subgraph sizes, let us make it explicit that the \emph{size of a graph $G$} is nothing but the quantity $e(G)$; of course, we shall also use the word `size' to denote the cardinality of a finite set but the precise meaning will always be clear from the context.
	
	The following classical result of Erd\H{o}s and Szemer\'{e}di~\citep{dense} about the edge densities of Ramsey graphs will prove to be useful.
	\begin{theorem}\label{density}
		The edge density of every $C$-Ramsey graph lies in the interval $[\delta, 1-\delta]$, where $0 < \delta \le 1/2$ is a constant depending on $C$ alone. \qed
	\end{theorem}
	
	For $c,\delta \in (0,1)$, an $n$-vertex graph $G$ is said to be \emph{$(c,\delta)$-diverse} if for each vertex $x\in V$, we have $|\Gamma(x) \triangle \Gamma(y)| \ge cn$ for all but at most $n^{\delta}$ vertices $y\in V$.  A Ramsey graph may not itself be diverse; however, the following result due to Bukh and Sudakov~\citep{diffdegrees} tells us that every Ramsey graph contains a large induced subgraph that is diverse.
	
	\begin{theorem}\label{diversity}
		For any $C, \delta > 0$, there exists a constant $c>0$ such that the following holds for all $n \in \N$. Every $C$-Ramsey graph on $n$ vertices contains a $(c,\delta)$-diverse induced subgraph on at least $cn$ vertices. \qed
	\end{theorem}
	
	We close this section by collecting together some standard results here for the sake of convenience; see~\citep{probmethod} for their proofs. First, we need the following classical estimate due to Tur\'{a}n.
	\begin{proposition}\label{iTuran}
		Every graph $G$ contains an independent set of size at least
		\[\frac{v(G)^2}{2e(G)+v(G)}.\eqno\qed\]
	\end{proposition}
	
	Next, we require the inequalities of Markov and Chebyshev.
	\begin{proposition}\label{iMarkov}
		Let $X$ be a non-negative real-valued random variable with mean $\mu$ and variance $\sigma^2$. For any $t\ge 0$, we have
		\[\PV(X>t)<\frac{\mu}{t}\]
		and
		\[\PV(|X-\mu|>t)<\frac{\sigma^2}{t^2}.\eqno\qed\]
	\end{proposition}
	
	We shall also make use of Hoeffding's inequality.
	\begin{proposition}\label{iHoeffding}
		Let $X_{1},X_2,\dots,X_{n}$ be independent real-valued random variables with $0\le X_{i}\le 1$ for each $1\le i\le n$. For any $t \ge 0$, writing $X=\sum_{i=1}^{n}X_{i}$, we have
		\[\PV(|X-\EV[X]|\ge t)\le 2\exp\left(\frac{-2t^{2}}{n}\right).\eqno\qed \]
	\end{proposition}
	
	The following tail bound for the hypergeometric distribution may be deduced from Hoeffding's inequality; see~\citep{hoeff} for a proof.
	\begin{proposition}\label{hypergeom}
		Let $X$ denote the number of successes in $D$ draws, without replacement, from a population of size $N$ that contains $M$ successes. For any $t \ge 0$, we have
		\[ \PV(|X - (M/N)D| \ge tD) \le 2\exp(-2t^2D). \eqno\qed\]
	\end{proposition}
	
	Finally, we need the following technical claim whose proof is a straightforward (if somewhat tedious) calculation using Proposition~\ref{hypergeom} and Stirling's approximation.
	\begin{proposition}\label{bineq}
		For any $\delta > 0$, there exists a constant $c >0$ such that the following holds for all $n \in \N$. Let $A$ and $B$ be disjoint subsets of $\{1,2, \dots, n\}$ such that $ \delta n  \le |A| \le (1-\delta)n$ and let $k$ be an integer satisfying $\delta n \le k \le n/2$.  If $U$ is a randomly chosen $k$-element subset of $\{1,2,\dots,n\}$, then
		\[\max_{r \in \Z} \left\{\PV (|U \cap B| - |U \cap A| = r)\right\} \le c/n^{1/2}.\]
	\end{proposition}
	\begin{proof}
		Let $B' = \{1, 2, \dots, n\} \setminus (A \cup B)$ and note that at least one of $B$ or $B'$ has cardinality at least $\delta n/2$.
		
		First, suppose that $|B| \ge |B'|$. In this case, we have $\delta/2 \le |B|/n \le 1- \delta$. Let $a = |A|$ and $X_a = |U\cap A|$, and similarly, let $b = |B|$ and $X_b = |U \cap B|$. Finally, let $s=|A\cup B| = a+b$ and $X = |U \cap (A \cup B)|= X_a + X_b$.
		
		Our task is to estimate $\max_{r\in \mathbb{Z}}\{\PV(X_b - X_a = r)\}.$ For a fixed $r \in \Z$, it is clear that
		\[
		\PV(X_b - X_a = r)=\sum_{t\in \mathbb{Z}}\PV(X_a=t, X_b=r+t).
		\]
		We have $\EV[X]=ks/n$ and as $\delta n \le k \le n/2$, we also have $\EV[X] \in [\delta s,s/2]$. Let $\C{T}_1 = \C{T}_1(r)$ denote the set of integers $t \in \Z$ satisfying $|\EV[X] - (2t+r)| \ge \delta s/10$ and let $\C{T}_2 = \C{T}_2(r) = \Z \setminus \C{T}_1$.
		
		We first estimate the sum
		\[
		T_1(r) = \sum_{t \in \C{T}_1}\PV(X_a=t, X_b=r+t).
		\]
		This sum is clearly bounded above by $\PV(|X-\EV[ X]|\ge \delta s /10)$. Since $X$ is a hypergeometric random variable, we deduce from Proposition~\ref{hypergeom} that
		\[T_1(r) \le \PV(|X-\EV [X]| \ge \delta s / 10) \le 2\exp( -\delta^3 n/50)\]
		for all $r \in \Z$.
		
		Next, we estimate the sum
		\[
		T_2(r) = \sum_{t \in \C{T}_2}\PV(X_a=t, X_b=r+t).
		\]
		To do so, note that
		\[
		\PV(X_a=t, X_b=r+t) =  \PV(X_a=t, X_b=r+t \,|\, X = 2t+r)\PV(X = 2t+r)
		\]
		and that
		\[
		\PV(X_a=t, X_b=r+t \,|\, X = 2t+r) =  \binom{a}{t}\binom{b}{t+r}\binom{s}{2t+r}^{-1}.
		\]
		It follows that
		\begin{align*}
			T_2(r) &\le \sum_{t \in \C{T}_2} \PV(X = 2t+r) \binom{a}{t}\binom{b}{t+r}\binom{s}{2t+r}^{-1}  \\
			&\le \max_{t \in \C{T}_2}  \left\{\binom{a}{t}\binom{b}{t+r}\binom{s}{2t+r}^{-1}\right\}.
		\end{align*}
		Since $|\EV[X]-(2t+r)| < \delta s /10$ for all $t \in \C{T}_2$ and $\EV[X] \in [\delta s, s/2]$, it is clear that
		\[(9\delta/10)s \le 2t+r \le (1/2 + \delta/10)s\]
		for all $t \in \C{T}_2$. Therefore, writing $\C{Z}$ for the set of pairs $(x,y) \in \Z^2$ with 
		\[(9\delta/ 10)s  \le x \le (1/2 + \delta/10)s \text{ and } 0\le y\le x,\] we have
		\[
		\max_{t \in \C{T}_2} \left\{  \binom{a}{t}\binom{b}{t+r}\binom{s}{2t+r}^{-1} \right\}\le
		\max_{(x,y) \in \C{Z}} \left\{ \binom{a}{y}\binom{b}{x-y}\binom{s}{x}^{-1}\right\}.\]
		It is easy to check using Stirling's approximation that
		\[\binom{a}{y}\binom{b}{x-y}\binom{s}{x}^{-1} \le 10\sqrt{\frac{s^{3}}{x(s-x)ab}}
		\]
		for all $0 \le y \le x$. As $a$, $b$ and $s$ are all contained in the interval $[\delta n/2, (1-\delta)n]$, there is a constant $c' > 0$ such that $T_2(r) \le c'/n^{1/2}$ for all $r \in \Z$. This proves the claim in the case where $|B| \ge |B'|$.
		
		The same argument as above, but with the roles of $B$ and $B'$ interchanged, proves the claim in the case where $|B'| \ge |B|$.
	\end{proof}
	
	\section{Overview of our strategy}\label{sketch}
	Let us first sketch our strategy to prove Theorem~\ref{mainthm}. Given a $C$-Ramsey graph $G = (V,E)$ on $n$ vertices, we shall show that $G$ contains induced subgraphs of many different sizes in three steps. 
	
	The first step, which is accomplished in Theorem~\ref{firststep}, is to show that for every $m \approx n^2$, it is possible to find a set of vertices $U$ with $|e(G[U]) - m| \approx n^{3/2}$ for which there is a set $W\subset V\setminus U$  of about $n^{1/2}$ vertices whose degrees in $U$ are all distinct and contained in some interval of length about $n^{1/2}$ located reasonably far away from the origin.
	
	The second step, which is the crux of the matter and accomplished in Theorem~\ref{secondstep}, is to show that the set $\{e(G[U\cup Z]): Z\subset W\}$ has size about $n^{3/2}$ for each pair $(U,W)$ obtained in the first step, and that these sizes are additionally all contained in an interval of length about $n^{3/2}$ centred at $e(G[U])$. The arguments in the second step are similar in spirit to those used in the first step, but much more involved. The added difficulty arises from having to exploit the two additional properties that are now at our disposal. First, we know that if $Z_1,Z_2 \subset W$ and $|Z_1|<|Z_2|$, then $e(G[U\cup Z_1]) < e(G[U\cup Z_2])$ since $U$ is much bigger than $W$ and each vertex of $W$ has many neighbours in $U$; this means that we may treat each `level' in the power set of $W$ independently. Second, we also know that the degrees $d_U(x)$ are distinct for all $x \in W$. We use these two facts to generate induced subgraphs of different sizes as follows. For many $k \approx n^{1/2}$, we randomly generate roughly $n^{1/2}$ different sets $Z \subset W$ with $|Z| = k$ for which the sums 
	\[ 
	\sum_{x \in Z} d_U(x) + e(G[Z]) 
	\] are `well-separated'. We then show how one may augment each such set $Z$ in about $n^{1/2}$ different ways by adding exactly one vertex to $U \cup Z$ from $W \setminus Z$; to do so, we use the fact that $G[W]$ is $C'$-Ramsey (for $C' \approx 2C$) to show that there are, on average, about $n^{1/2}$ vertices in $W\setminus Z$ all with distinct degrees in $U\cup Z$. 
	
	The third and final step consists of putting together the $\approx n^{3/2}$ subgraph sizes obtained at $\approx n^{1/2}$ scales together to yield $\approx n^{2}$ different subgraph sizes; of course, some care is needed to ensure that no two subgraphs at `different scales' have the same size, but this is for the most part straightforward.
	
	\section{Proof of the main result}\label{proof}
	We begin with the following claim that makes precise the first step outlined in our strategy in the previous section.
	\begin{theorem}\label{firststep}
		For any $C, \eps > 0$, there exist constants $c_1,c_2>0$ such that the following holds for all sufficiently large $n \in \N$. If $G=(V,E)$ is a $C$-Ramsey graph on $n$ vertices, then for every integer $m$ satisfying $c_{1}n^{2} \le m \le 2 c_{1}n^{2}$, there exist disjoint sets $U, W \subset V$ such that
		\begin{enumerate}
			\item\label{1:1} $m-2n^{3/2}  \le e(G[U]) \le m+2n^{3/2}$,
			\item\label{1:2} $|W| \ge n^{1/2-\eps}$,
			\item\label{1:3} there exists a positive real number $l \ge c_{2}n$ such that $l \le \Deg_{U}(x) \le l+n^{1/2+\eps}$ for all $x\in W$, and
			\item\label{1:4} the degrees $\Deg_{U}(x)$ are distinct for all $x\in W$.
		\end{enumerate}
	\end{theorem}
	
	\begin{proof}
		In what follows, the constants we define may depend on $C$ and $\eps$, but will never depend on $n$. Fix a positive constant $\delta  = \eps/2$. Applying Theorem~\ref{diversity} to our graph $G$, we find a subset $V'\subset V$ with $|V'| \ge c|V|$ for which $G[V']$ is $(c,\delta)$-diverse for some positive constant $c$ depending on $C$ and $\eps$ alone. We shall henceforth work exclusively with $G[V']$; in particular, all vertex degrees and neighbourhoods in what follows are with respect to this graph.
		
		Note that $G[V']$ is $2C$-Ramsey for all sufficiently large $n$, so by Theorem~\ref{density}, we know that the edge density of this graph is bounded below by a constant $\alpha>0$ depending on $C$ and $\eps$ alone. A simple averaging argument tells us that the degree of at least $\alpha |V'|/2$ vertices from $V'$ is greater than or equal to $\alpha |V'|/2$; let the set of such vertices be $W_{0}$. By the pigeonhole principle, there exists a positive integer $l_{0} \ge \alpha|V'|/2 \ge \alpha cn/2$ and a subset $W_{1} \subset W_{0}$ of size $n^{1/2}$ such that the degree of every $w\in W_{1}$ lies in the interval $[l_{0},l_{0}+(2n^{1/2}/\alpha c)]$.
		
		We now define $c_{1}=\alpha c^{2}/32$ and $c_2 = \alpha^{3/2}c^2/12$. Let $s=e(G[V'\setminus W_{1}])$ and note that
		\[ s \ge \alpha \binom{|V'|}{2} - |W_1|n \ge \alpha c^{2}n^{2}/4.\]
		
		Now, let $m$ be a positive integer satisfying $c_{1}n^{2} \le m \le 2c_{1}n^{2}$. We define $p = p(m) =(m/s)^{1/2}$ and note that  $c^{1/2}_1 \le p \le 1/2$. Select $U \subset V'\setminus W_{1}$ randomly by selecting each vertex of $V'\setminus W_{1}$ with probability $p$, independently of the other vertices. We shall show that with positive probability, there is a subset $W$ of $W_{1}$ such that the sets $U$ and $W$ satisfy the conditions of the theorem.
		
		We first deal with condition~(\ref{1:1}). Let $\C{E}_1$ denote the event that we have $|e(G[U])-m| \le 2n^{3/2}$. We prove the following claim.
		\begin{claim}
			$\PV(\C{E}_1) \ge 3/4$.
		\end{claim}
		\begin{proof}
			We bound the probability of $\C{E}_1$ using the second moment of $X=e(G[U])$. First, as $U$ contains every edge of $G[V'\setminus W_{1}]$ with probability $p^{2}=m/s$, we have $\EV[X]=m$. We claim that $\sigma_X^{2} \le n^3$. To see this, first write $X=\sum_{a\in F} I(a)$, where $F = E(G[V'\setminus W_{1}])$ and $I(a)$ is the indicator of the event $\{a\in E(G[U])\}$, and then note that
			\[\EV [X^{2}]=\sum_{(a,b)\in F^{2}}\EV [(I(a)I(b)]=\sum_{|a\cap b|=0}p^{4}+\sum_{|a\cap b|=1}p^{3}+\sum_{a}p^{2},\]
			and
			\[\EV [X]^{2}=\sum_{(a,b)\in F^{2}}\EV[I(a)]\EV[I(b)]=\sum_{(a,b)\in F^{2}}p^{4}.\]
			Hence,
			\[\sigma_X^{2}=\sum_{|a\cap b|=1}(p^{3}-p^{4})+\sum_{a}(p^{2}-p^{4}) \le n^{3}(p^{3}-p^{4})+n^{2}(p^{2}-p^{4}) \le n^{3}.\]
			It now follows from Chebyshev's inequality that
			\[\PV(\C{E}_1)=\PV(|X-m| \le 2n^{3/2}) \ge 3/4. \qedhere\]
		\end{proof}
		
		Next, we address condition~(\ref{1:3}). Let $t=n^{1/2+\delta}/4$ and let $\C{E}_2$ denote the event that for all $x\in W_{1}$, we have $|\Deg_{U}(x)-pl_{0}| \le 2t$. We have the following bound for the probability of $\C{E}_2$.
		\begin{claim}
			$\PV(\C{E}_2) \ge 3/4$.
		\end{claim}
		\begin{proof}
			For $x\in V'\setminus W_{1}$, let $I(x)$ be the indicator of the event $\{x\in U\}$. Clearly, $ \Deg_{U}(x)= \sum_{y\in \Gamma(x)}I(y)$ and since these indicators are independent, we deduce from Hoeffding's inequality that
			\[ \PV(|\Deg_{U}(x)-\EV [\Deg_{U}(x)]|>t) \le 2\exp\left(-\frac{2t^{2}}{\Deg_U(x)}\right).\]
			For $x \in W_1$, we have $\EV[\Deg_{U}(x)]=p\Deg_{V'\setminus W_{1}}(x)$ and $|pl_{0}-\EV [ \Deg_{U}(x)]| \le n^{1/2} + (2n^{1/2}/\alpha c)$, where the last inequality holds since $W_1$ contains precisely $n^{1/2}$ vertices.
			Therefore, for all sufficiently large $n$, we have
			\[\PV(|\Deg_{U}(x)-pl_{0}|>2t) \le \PV(|\Deg_{U}(x)-\EV[ \Deg_{U}(x)]|> t) \le 2\exp(-n^{-2\delta}/8)\]
			for all $x \in W_1$. Thus, if $n$ is sufficiently large, it follows from the union bound that
			\[\PV(\C{E}_2) \ge 1 - 2n\exp(-n^{-2\delta}/8) \ge 3/4.\qedhere\]
		\end{proof}
		
		Finally, we deal with conditions~(\ref{1:2}) and~(\ref{1:4}). To do so, we define an auxiliary \emph{degree graph} $D$ on $W_{1}$ where two vertices $x,y\in W_{1}$ are joined by an edge if  $\Deg_{U}(x)=\Deg_{U}(y)$. Writing  $\C{E}_3$ for the event that we have  $e(D) \le 8n^{\delta}|W_{1}|$, we have the following.
		
		\begin{claim}
			$\PV(\C{E}_3) \ge 3/4$.
		\end{claim}
		\begin{proof}
			For a pair of vertices $x,y\in W_{1}$, let $I(x,y)$ be the indicator of the event $\{ \Deg_{U}(x)=\Deg_{U}(y) \}$. Let us estimate the expected number of edges of $D$. Say that a pair $\{x,y\}\subset W_{1}$ is \emph{good} if $|\Gamma(x)\triangle \Gamma(y)| \ge c|V'|$, and \emph{bad} otherwise.
			
			We claim that there exists a constant $c'>0$ such that $\PV(I(x,y)=1) \le c'/n^{1/2}$ for any good pair $\{x,y\} \subset W_1$. To see this, fix a good pair $\{x,y\} \subset W_1$ and let $a=|\Gamma(x)\setminus \Gamma(y)|$ and $b=|\Gamma(y)\setminus \Gamma(x)|$. Without loss of generality, assume that $a\le b$. As $a+b \ge cn$, we have $b\ge cn/2$. Note that $\Deg_{U}(x)=\Deg_{U}(y)$ if and only if $|(\Gamma(x)\setminus \Gamma(y))\cap U|=|(\Gamma(y)\setminus \Gamma(x))\cap U|$. It is now easy to see that
			\begin{align*}
				\PV(\Deg_{U}(x)=\Deg_{U}(y)) &=\sum_{i=0}^{a}p^{i}(1-p)^{a-i}\binom{a}{i}p^{i}(1-p)^{b-i}\binom{b}{i} \\
				&\le \max_{0 \le i \le b} \left\{p^{i}(1-p)^{b-i}\binom{b}{i} \right\} <\frac{10}{\sqrt{p(1-p)b}},
			\end{align*}
			where the last inequality is an easy consequence of Stirling's approximation for the factorial. As $b \ge cn/2$ and $0<c_{1}^{1/2} \le p \le 1/2$, it is clear that there exists a constant $c'>0$ as claimed.
			
			As $V'$ is $(c,\delta)$-diverse, the number of bad pairs from $W_{1}$ is at most $|W_{1}|n^{\delta}$. Hence, for all sufficiently large $n$, we have
			\[\EV[e(D)]=\sum_{\{x,y\} \subset W_1}\PV(I(x,y)=1) \le |W_{1}|n^{\delta}+c'\frac{|W_{1}|^{2}}{n^{1/2}} \le 2|W_{1}|n^{\delta}.\]
			It now follows from Markov's inequality that
			\[\PV(\C{E}_3)=\PV(e(D) \le 8|W_{1}|n^{\delta}) \ge \PV(e(D) \le 4\EV[ e(D)]) \ge 3/4.\qedhere\]
		\end{proof}
		
		Let $W \subset W_1$ be an independent set of maximum size in the graph $D$. We claim that the sets $U$ and $W$ satisfy the conditions of the theorem with positive probability. To see this, first note that $\PV (\C{E}_1 \cap \C{E}_2 \cap \C{E}_3) \ge 1/4$. Now, if $\C{E}_1$ holds, then condition~\eqref{1:1} is clearly satisfied. Next, if $\C{E}_2$ holds, then since $W$ is a subset of $W_{1}$, condition~\eqref{1:3} holds with $l=pl_{0}-n^{1/2+\delta}/2 \ge c_2 n$. Finally, if $\C{E}_3$ holds, then the degree graph $D$ has at most $8|W_{1}|n^{\delta}$ edges. Applying Tur\'{a}n's theorem, we see that $D$ has an independent set of size at least $|W_{1}|/(16n^{\delta}+1) \ge n^{1/2 - \eps}$. As $W \subset W_1$ is an independent set of maximum size in $D$,  it is clear that $|W| \ge n^{1/2 - \eps}$ and that the degrees $\Deg_{U}(x)$ are distinct for all $x\in W$, so conditions~\eqref{1:2} and~\eqref{1:4} are also satisfied.
	\end{proof}
	
	Before we proceed further, it will help to have some notation. By a \emph{weighted graph}, we mean a graph $G=(V,E)$ with a weight function $\omega\colon  V\rightarrow \mathbb{N} \cup \{0\}$. Define the \emph{$\omega$-size} of $G$ by $e^{\omega}(G)=e(G)+\omega(V)$, where $\omega(U)=\sum_{v\in U}\omega(v)$ for any $U \subset V$. The \emph{$\omega$-degree} of a vertex $v\in V$ is given by $\Deg^{\omega}(v)=\Deg(v)+\omega(v)$; also, for a subset of vertices $U \subset V$, let $\Deg_{U}^{\omega}(v)=\Deg_{U}(v)+\omega(v)$ denote the $\omega$-degree of $v$ in $U$. We finally define, as before, the set
	\[\Phi(G, \omega)=\{e^{\omega}(H): H \text{ is an induced subgraph of }G \}\]
	and the set
	\[ \Psi(G,\omega)=\{(v(H), e^\omega(H)): H \text{ is an induced subgraph of }G \}.\]
	
	Turning to the second step in the proof of Theorem~\ref{mainthm}, we fix a pair of subsets $(U,W)$ as in Theorem~\ref{firststep} and restrict our attention to the subgraphs induced by sets of the form $U\cup Z$ for some $Z\subset W$. Our aim is to show that the set ${\{e(G[U\cup Z]):Z\subset W\}}$ contains about $n^{3/2}$ elements. As $U$ will stay fixed, it will be more convenient to attach a weight of $\Deg_{U}(v)$ to each vertex $v\in W$. Therefore, let $H = G[W]$ and define a weight function $\omega$ on $W$ by $\omega(v)=\Deg_{U}(v)$. As
	\[e(G[U\cup Z])=e(G[U])+e^{\omega}(G[Z]),\]
	the set $\Phi(H, \omega)$ is just a translate of the set $\{e(G[U\cup Z]):Z\subset W\}$. Now, observe that $H = G[W]$ is $(2C+2\eps)$-Ramsey and note also that since the weights on the vertices of $H$ are `large', subsets of $W$ of different sizes induce subgraphs of $H$ of different $\omega$-sizes. These observations lead to the formulation of the following result, which is the main ingredient in the proof of Theorem~\ref{mainthm}.
	
	\begin{theorem}\label{secondstep}
		For any $C,\delta > 0$, the following holds for all $\eps>4\delta$ and  all sufficiently large $n \in \N$. If $G$ is a $C$-Ramsey graph on $n$ vertices with an injective weight function
		\[\omega\colon V(G)\rightarrow \{0, 1, \dots,n^{1+\delta}\},\]
		then the set $\Psi(G,\omega)$ contains at least $n^{3-\eps}$ elements.
	\end{theorem}
	
	Let us prepare for the proof of this theorem with the following simple lemma, which is a bipartite version of Theorem~\ref{diversity}.
	
	\begin{lemma}\label{bipartitediverse}
		For any $C, \delta > 0$, there exists a constant $c>0$ such that the following holds for all sufficiently large $n \in \N$. If $G$ is a $C$-Ramsey graph on $n$ vertices, then there exist disjoint sets $X,Y\subset V(G)$ with $|X|,|Y| \ge cn$ such that for every $u\in Y$, $|\Gamma_{X}(u)\  \triangle \Gamma_{X}(v)| \ge cn$ for all but at most $n^{\delta}$ vertices $v \in Y$.
	\end{lemma}
	\begin{proof}
		By Theorem~\ref{diversity}, there exists a subset of $c'n$ vertices $W\subset V$ such that $G[W]$ is $(c',\delta)$-diverse, where $c'>0$ is a constant depending on $C$ and $\delta$ alone. We construct $X$ and $Y$ randomly from $W$ by assigning each vertex of $W$ uniformly at random to either $X$ or $Y$, independently of the other vertices.
		
		Let $\C{E}_1$ be the event that $|X|,|Y| \ge|W|/3$. It is immediate from Hoeffding's inequality that $\PV(\C{E}_1) \ge 3/4$ provided $n$ is sufficiently large. Let $\C{E}_2$ be the event that for all $u, v\in W$ satisfying $|\Gamma(u) \triangle \Gamma(v)|\ge c'n$, we have $|\Gamma_{X}(u)\triangle \Gamma_{X}(v)| \ge c'n/3$. Again, it is easy to deduce from Hoeffding's inequality and a simple union bound that $\PV(\C{E}_2) \ge 3/4$ for all sufficiently large $n$. Thus, $\PV(\C{E}_1 \cap \C{E}_2) \ge 1/2$, proving the claim with $c= c'/3$.
	\end{proof}
	
	We are now ready to prove Theorem~\ref{secondstep}.
	
	\begin{proof}[Proof of Theorem~\ref{secondstep}]
		In what follows, all inequalities will hold provided $n$ is sufficiently large. We apply Lemma~\ref{bipartitediverse} to our $n$-vertex $C$-Ramsey graph $G = (V,E)$ to find two disjoint sets $X,Y \subset V$, each of size at least $cn$, such that for each vertex $u\in Y$, there are at most $n^{\delta}$ other vertices $v\in Y$ for which $|\Gamma_{X}(u)\triangle \Gamma_{X}(v)|<cn$; here, $c>0$ is a constant that depends on $C$ and $\delta$ alone.
		
		Fix $m=|X|/2-cn/4$ and note that $m \ge cn/4$. Order the vertices of $X$ in increasing order of weight and let $S$ and $T$ respectively denote the first and last $m$ vertices in this ordering of $X$. Observe that since $\omega$ is injective, we have
		\[\omega(u)-\omega(v) \ge |X|-2m=cn/2 \] for all $u \in T$ and $v\in S$. Finally, let $X' = X \setminus (S\cup T)$ and note that $|X'| = cn/2$.
		
		Say that a vertex $y\in Y$ is of
		\begin{enumerate}
			\item type 1 if $\Deg_{S}(y)<cn/8$,
			\item type 2 if $\Deg_{S}(y)>m-cn/8$,
			\item type 3 if $\Deg_{T}(y)<cn/8$,
			\item type 4 if $\Deg_{T}(y)>m-cn/8$, and
			\item type 0 if it is not of any of the previous types.
		\end{enumerate}
		A vertex may have more than one (non-zero) type; call a vertex $y\in Y$ \emph{problematic} if it has two different types. We first make the following observation.
		
		\begin{claim}
			There are at most $4n^{\delta}$ problematic vertices in $Y$.
		\end{claim}
		\begin{proof}
			If $v\in Y$ is problematic, then there are four possibilities: either $v$ is of types 1 and 3, 1 and 4, 2 and 3, or 2 and 4. If $u,v\in Y$ are both problematic vertices of the same type, then
			\begin{align*}
				|\Gamma_{X}(u)\triangle \Gamma_{X}(v)|&=|\Gamma_{S}(u)\triangle \Gamma_{S}(v)|+|\Gamma_{T}(u)\triangle \Gamma_{T}(v)|+|\Gamma_{X'}(u)\triangle \Gamma_{X'}(v)|\\
				&<cn/4+cn/4+|X'|=cn.
			\end{align*}
			Hence, for each valid pair of types, there are at most $n^{\delta}$ problematic vertices of those types, proving the claim.
		\end{proof}
		
		It follows that $Y$ contains at least $cn-4n^{\delta} \ge cn/2$ non-problematic vertices; let $Y'$ be this set of vertices. Now, there is a subset $Y''\subset Y'$ of size at least $|Y'|/5 \ge cn/10$ where every vertex has the same type. By the pigeonhole principle, there is a subset $Z\subset Y''$ with $|Z| \ge c^2n^{1-\delta}/40$ for which the set ${\{\omega(z):z\in Z\}}$ is contained an interval of size $cn/4$.
		
		Let $\C{I} \subset \N^2$ be the set of pairs $(k,i)$ that satisfy $cn/20 \le k \le cn/10$ and  $cn/50 \le i \le cn/25$. In order to exhibit many distinct elements in $\Psi(G, \omega)$, we shall construct a family of random subsets of $X$, one for each element of $\C{I}$.
		
		For each integer $cn/20 \le k \le cn/10$, let $S_k$ be a uniformly random ordering of a $k$-element subset of $S$ chosen uniformly at random and let $T_k$ similarly be a random ordering of a random $k$-element subset of $T$, with $S_k$ and $T_k$ being chosen independently of each other; in what follows, we write $S_k= (x_{k,1}, x_{k,2}, \dots,x_{k,k})$  and $T_k =(x_{k,k+1}, x_{k,k+2}, \dots,x_{k,2k})$. Now, for each $(k,i) \in \C{I}$, define the random set $L(k,i)=\{x_{k,i+1},x_{k,i+2},\dots,x_{k,i+k}\}$ and the random variable
		\[N(k,i)=|\{\Deg^{\omega}_{L(k,i)}(v):v\in Z\}|.\] 
		Finally, to simplify notation, for each $(k,i) \in \C{I}$ and each $v \in Z$, let $\wdeg_{k,i}(v)$ denote the random variable $\Deg^{\omega}_{L(k,i)}(v)$, i.e., the $\omega$-degree of $v$ in $L(k,i)$.
		
		Our starting point is the following.
		\begin{claim}\label{lowerbound}
			The cardinality of $\Psi(G,\omega)$ is at least $N=\sum_{(k,i)\in\C{I}}N(k,i)$.
		\end{claim}
		\begin{proof}
			Consider the set
			\[\psi =\{(k+1,e^{\omega}(G[L(k,i)\cup \{z\}])): (k,i)\in \C{I},z\in Z\}\]
			and note that $\psi \subset\Psi(G,\omega)$. Therefore, it is enough to prove that $|\psi|=N$. Fix $(k,i) \in \C{I}$ and $z \in Z$ and suppose that
			\[(k+1,e^{\omega}(G[L(k,i)\cup \{z\}]))=(k'+1,e^{\omega}(G[L({k',i'})\cup \{z'\}]))\]
			for some $(k',i') \in \C{I}$ and $z' \in Z$. Our claim will follow if we show that this is only possible when $k=k'$, $i=i'$ and $\wdeg_{k,i}(z)=\wdeg_{k,i}(z')$.
			
			Trivially, we must have $k=k'$. Now suppose that $i \le i'$ and write $L = L(k,i)$ and $L' = L(k,i')$. Since $|L\setminus L'|=i'-i$, note that
			\[e(G[L])-e(G[L']) \le k(i'-i) \le c(i'-i)n/10.\]
			It is also clear that
			\[ \wdeg_{k,i}(z) - \wdeg_{k,i'}(z') = \Deg_{L}(z) - \Deg_{L'}(z') + \omega(z) - \omega(z') \le cn/10 + cn/4.\]
			Finally, since $L' \setminus L \subset T$ and $L\setminus L' \subset S$, it follows that
			\[\omega(L')-\omega(L) \ge c(i'-i)n/2.\]
			Putting these three inequalities together, it follows that if $i < i'$, then
			\[
			e^{\omega}(G[L'\cup \{z'\}])-e^{\omega}(G[L\cup \{z\}]) \ge cn((i'-i)(1/2 -1/10) - (1/4 + 1/10))>0.
			\]
			Thus, it must be the case that $i=i'$. It is then clear that  $\Deg^{\omega}_{L}(z)=\Deg^{\omega}_{L}(z')$, proving the claim.
		\end{proof}
		
		By the previous claim, it suffices to show that with positive probability, we have \[\sum_{(k,i)\in\C{I}} N(k,i)>n^{3-\eps}.\] We shall deduce this as a consequence of the following two observations.
		\begin{enumerate}
			\item First, if $|\omega(u)-\omega(v)|$ is large for a pair of vertices $u, v \in Z$, then the difference between their expected $\omega$-degrees in $L(k,i)$ cannot be small for too many pairs $(k,i) \in \C{I}$.
			\item Next, if $u$ and $v$ are vertices in $Z$ for which the difference between their expected $\omega$-degrees in $L(k,i)$ is large, then the probability that we have  $\wdeg_{k,i}(u)=\wdeg_{k,i}(v)$ is extremely small.
		\end{enumerate}
		
		To make these observations precise, we need a few definitions. For $v \in Z$ and $(k,i) \in \C{I}$, let
		\[\mu_{k,i}(v) = \EV \left[\wdeg_{k,i}(v)\right] = \EV\left[ \Deg^{\omega}_{L(k,i)}(v) \right].\]
		To determine this expectation, observe that $L(k,i)\cap S$ is uniformly distributed on the $(k-i)$-element subsets of $S$ and $L(k,i)\cap T$ is uniformly distributed on the $i$-element subsets of $T$, so
		\[\mu_{k,i}(v)=\omega(v)+\left(\frac{k-i}{m}\right)\Deg_{S}(v)+\left(\frac{i}{m}\right)\Deg_{T}(v)\]
		for all $v \in Z$ and $(k,i) \in \C{I}$. Say that $(k,i) \in \C{I}$ is \emph{compatible} with a pair of vertices $\{u,v\} \subset Z$ if
		\[|\mu_{k,i}(u)-\mu_{k,i}(v)| \ge n^{1/2+\delta};\]
		otherwise, we say that $(k,i)$ is \emph{incompatible} with $\{u,v\}$.
		
		The following claim makes our first observation precise.
		
		\begin{claim}\label{badpairs}
			For any $u, v \in Z$, the number of pairs $(k,i)\in \C{I}$ incompatible with $\{u,v\}$ is at most $4n^{5/2+\delta}/|\omega(u)-\omega(v)|.$
		\end{claim}
		
		\begin{proof}
			If $|\omega(u)-\omega(v)|<2n^{1/2+\delta}$, then $4n^{5/2+\delta}/|\omega(u)-\omega(v)|>n^{2}$, in which case the claim is trivial since $|\C{I}|<n^{2}$. Therefore, we may assume without loss of generality that $|\omega(u)-\omega(v)|\ge 2n^{1/2+\delta}$. Let $\Delta_1=\omega(u)-\omega(v)$, $\Delta_2 = \Deg_{S}(u)-\Deg_{S}(v)$ and $\Delta_3 = \Deg_{S}(v)-\Deg_{S}(u)+\Deg_{T}(u)-\Deg_{T}(v)$; clearly,
			\[\mu_{k,i}(u)-\mu_{k,i}(v)=\Delta_1+\frac{k}{m}\Delta_2+\frac{i}{m}\Delta_3\]
			for all $(k,i) \in \C{I}$. Now, since $k,i \le m/2$ for all $(k,i)\in\C{I}$ and $|\Delta_1|\ge 2n^{1/2+\delta}$, there exists a pair $(k,i) \in \C{I}$ for which
			\[\left|\Delta_1+\frac{k}{m} \Delta_2 +\frac{i}{m} \Delta_3 \right|<n^{1/2+\delta}\]
			only if either $|\Delta_2|>|\Delta_1|/2$ or $|\Delta_3|>|\Delta_1|/2$.
			
			First, suppose that $|\Delta_2|>|\Delta_1|/2$. Fix $i_{0} \in \N$ and consider the set $K(i_{0})$ of those integers $k$ for which the pair $(k,i_{0}) \in \C{I}$ is incompatible with $\{u,v\}$. Let $k_{0}$ be the smallest element of $K(i_{0})$. If $k>k_{0}+4n^{3/2+\delta}/|\Delta_1|$, then \[\left|\left(\Delta_1+\frac{k}{m}\Delta_2+\frac{i_{0}}{m} \Delta_3 \right) - \left(\Delta_1 +\frac{k_{0}}{m} \Delta_2 +\frac{i_{0}}{m}\Delta_3 \right)\right|>\frac{4|\Delta_2|n^{3/2+\delta}}{|\Delta_1|m}>2n^{1/2+\delta}.\]
			Since $|\Delta_1+k_{0}\Delta_2/m+  i_{0}\Delta_3/m|<n^{1/2+\delta}$, we must have $|\Delta_1+k\Delta_2/m+  i_{0}\Delta_3/m| \ge n^{1/2+\delta}$, so $k \notin K(i_0)$. It follows that
			\[K(i_{0})\subset \left[k_{0},k_{0}+\frac{4n^{3/2+\delta}}{|\Delta_1|}\right].\]
			Consequently, the number of pairs $(k,i) \in \C{I}$ incompatible with $\{u,v\}$ is at most $4n^{5/2+\delta}/|\omega(u)-\omega(v)|$, as claimed.
			
			If $|\Delta_3|>|\Delta_1|/2$, then a similar argument (with the roles of $k$ and $i$ interchanged) establishes the claim.
		\end{proof}
		
		For each $(k,i) \in \C{I}$, define a graph $H_{k,i}$ on $Z$ where two vertices $u$ and $v$ are joined by an edge if $(k,i)$ is incompatible with $\{u,v\}$. The following is an easy corollary of the previous claim.
		
		\begin{claim}\label{totalbadpairs}$\sum_{(k,i)\in \C{I}} e(H_{k,i})<10n^{7/2+\delta}\log n.$
		\end{claim}
		
		\begin{proof} It follows from Claim~\ref{badpairs} that
			\begin{align*}
				\sum_{(k,i)\in \C{I}} e(H_{k,i})&=\sum_{\{u,v\}\subset Z}|\{(k,i)\in\C{I}:(k,i) \mbox{ is incompatible with } \{u,v\}\}| \\
				&\le \sum_{\{u,v\} \subset Z}\frac{4n^{5/2+\delta}}{|\omega(u)-\omega(v)|}.
			\end{align*}
			Let $z_{1}, z_2, \dots,z_{|Z|}$ be the elements of $Z$ with $\omega(z_{1})< \omega(z_2) < \dots<\omega(z_{|Z|})$. As $\omega$ is integer-valued, we have $|\omega(z_{a})-\omega(z_{b})|\ge |a-b|$. Hence,
			\[\sum_{\{u,v\}\subset Z}\frac{4n^{5/2+\delta}}{|\omega(u)-\omega(v)|}\le |Z|\sum_{j=1}^{|Z|}\frac{8n^{5/2+\delta}}{j}<10n^{7/2+\delta}\log n. \qedhere \]
		\end{proof}
		
		Next, for each $(k,i)\in\C{I}$, define an auxiliary graph $D_{k,i}$ on $Z$ where two vertices $u$ and $v$ are joined by an edge if $\wdeg_{k,i}(u)=\wdeg_{k,i}(v)$. It is clear that $N(k,i)$ is the size of the largest independent set in $D_{k,i}$, so by Tur\'{a}n's theorem, we have
		\[N(k,i)\ge\frac{|Z|^{2}}{2e(D_{k,i})+|Z|}.\]
		Consequently, we deduce using convexity that
		\[
		\sum_{(k,i)\in\C{I}}N(k,i)\ge \sum_{(k,i)\in\C{I}}\frac{|Z|^{2}}{2e(D_{k,i})+|Z|}\ge |\C{I}|^{2}|Z|^{2}\left(\sum_{(k,i)\in\C{I}}(2e(D_{k,i})+|Z|)\right)^{-1}.
		\]
		It follows that
		\begin{equation}\label{equ1}
			\sum_{(k,i)\in\C{I}}N(k,i) \ge 10^{-8}c^{6}n^{6-2\delta}\left(n^3+\sum_{(k,i)\in\C{I}}2e(D_{k,i})\right)^{-1}\tag{\textasteriskcentered}
		\end{equation}
		because $ 10^{-3}c^{2}n^{2} \le |\C{I}| \le n^2$ and $cn^{1-\delta}/4 \le |Z| \le n$. The last ingredient in our proof is the following claim, which says that the right-hand side of~(\ref{equ1}) is large on average.
		
		\begin{claim}\label{important}$\sum_{(k,i)\in\C{I}}\EV[e(D_{k,i})]<n^{3+2\delta}.$
		\end{claim}
		
		Note that
		\[ \sum_{(k,i)\in\C{I}}\EV[e(D_{k,i})]=\sum_{(k,i)\in\C{I}}\sum_{\{u,v\}\subset Z}\PV(\wdeg_{k,i}(u)=\wdeg_{k,i}(v)).\]
		To prove Claim~\ref{important}, it will therefore be convenient to have bounds for probability of the event $\{ \wdeg_{k,i}(u)=\wdeg_{k,i}(v) \}$ for various $(k,i) \in \C{I}$ and $\{u,v\} \subset Z$. First, we prove the following bound that handles the case where $(k,i)$ is compatible with $\{u,v\}$.
		\begin{claim}\label{subclaim1}
			If $(k,i) \in \C{I}$ is compatible with $\{u,v\} \subset Z$, then
			\[\PV(\wdeg_{k,i}(u)=\wdeg_{k,i}(v))<4\exp(-n^{\delta}).\]
		\end{claim}
		
		\begin{proof} To simplify notation, let $a=\wdeg_{k,i}(u)$ and $b=\wdeg_{k,i}(v)$. Since $(k,i)$ is compatible with $\{u,v\}$, we have $|\EV[a]-\EV[b]|\ge n^{1/2+\delta}.$ Hence,
			\[\PV(a=b)\le\PV\left(|a-\EV[a]| \ge \frac{n^{1/2+\delta}}{2}\right)+\PV\left(|b-\EV[b]| \ge \frac{n^{1/2+\delta}}{2}\right).\]
			Let $a_s=\Deg_{L(k,i)\cap S}(u)$ and $a_t=\Deg_{L(k,i)\cap T}(u)$. Since $a=a_s+a_t+\omega(u)$, we see that
			\[\PV\left(|a-\EV[a]| \ge \frac{n^{1/2+\delta}}{2}\right) \le\PV\left(|a_s-\EV[a_s]| \ge \frac{n^{1/2+\delta}}{4}\right)+\PV\left(|a_t-\EV[a_t]| \ge \frac{n^{1/2+\delta}}{4}\right).\]
			The distribution of $a_t$ is hypergeometric, so it is easy to check using Proposition~\ref{hypergeom} that
			\[\PV\left(|a_t-\EV[a_t]| \ge \frac{n^{1/2+\delta}}{4}\right) < \exp(-n^{\delta}).\]
			Similarly, we also have $\PV(|a_s-\EV[ a_s]|>n^{1/2+\delta}/4)<\exp(-n^{\delta})$. Therefore,
			\[\PV\left(|a-\EV[a]| \ge \frac{n^{1/2+\delta}}{2}\right)<2\exp(-n^{\delta}).\]
			It is clear that we also have $\PV(|b-\EV [b]|\ge n^{1/2+\delta}/2)<2\exp(-n^{\delta})$. We conclude that $P(a=b)<4\exp(-n^{\delta})$, as claimed.
		\end{proof}
		
		We say that a pair $\{u,v\} \subset Z$ is \emph{good} if $|\Gamma_{X}(u)\triangle \Gamma_{X}(v)|\ge cn$, and \emph{bad} otherwise. The next claim handles the case where $(k,i)$ is incompatible with $\{u,v\}$.
		\begin{claim}\label{subclaim2}
			If $\{u,v\} \subset Z$ is a good pair and $(k,i) \in \C{I}$ is incompatible with $\{u,v\}$, then
			\[\PV(\wdeg_{k,i}(u)=\wdeg_{k,i}(v))<c'/n^{1/2},\]
			where $c'>0$ is a constant depending on $C$ and $\delta$ alone.
		\end{claim}
		
		\begin{proof}Let $P=L(k,i)\cap S$ and $Q=L(k,i)\cap T$. Clearly, $P$ is uniformly distributed on the $(k-i)$-element subsets of $S$ and similarly, $Q$ is uniformly distributed on the $i$-element subsets of $T$; furthermore, $P$ and $Q$ are independent.
			
			To prove the claim, it will be convenient to define the sets $A_{1}=\Gamma_{S}(u)\setminus \Gamma_{S}(v)$, $A_{2}=\Gamma_{S}(v)\setminus \Gamma_{S}(u)$, $B_{1}=\Gamma_{T}(u)\setminus \Gamma_{T}(v)$ and $B_{2}=\Gamma_{T}(v)\setminus \Gamma_{T}(u)$. Also, let $\alpha_{j}=|A_{j}\cap P|$ and $\beta_{j}=|B_{j}\cap Q|$ for $j \in \{1,2\}$.
			
			Observe that \[\wdeg_{k,i}(u)-\wdeg_{k,i}(v)=\alpha_{1}+\beta_{1}-\alpha_{2}-\beta_{2}+ \Delta,\]
			where $\Delta=\omega(u)-\omega(v)$. Hence, $\wdeg_{k,i}(u)=\wdeg_{k,i}(v)$ if and only if $\alpha_{1}-\alpha_{2}+\Delta=\beta_{2}-\beta_{1}$, so it consequently follows that
			\begin{align*}
				\PV(\wdeg_{k,i}(u)=\wdeg_{k,i}(v))&=\sum_{r\in \mathbb{Z}}\PV(\alpha_{1}-\alpha_{2}+ \Delta=\beta_{2}-\beta_{1}=r)\\
				&=\sum_{r\in\mathbb{Z}}\PV(\alpha_{1}-\alpha_{2}=r - \Delta)\PV(\beta_{2}-\beta_{1}=r)\\
				&\le\min\left\{\max_{r\in \mathbb{Z}}\left\{\PV(\alpha_{2}-\alpha_{1}=r)\right\}, \max_{r\in \mathbb{Z}}\left\{\PV(\beta_{2}-\beta_{1}=r)\right\}\right\}.
			\end{align*}
			
			Suppose first that $u$ and $v$ are both either of type~1 or~2. In this case, we would like to estimate $\max_{r\in \mathbb{Z}}\{\PV(\beta_{2}-\beta_{1}=r)\}$ using Proposition~\ref{bineq}. Recall that $m \ge cn/4$ and that $i \ge cn/50$. Therefore, to apply Proposition~\ref{bineq}, it suffices to show that either $b_1/m$ or $b_2/m$ is bounded away from both $0$ and $1$, where $b_j=|B_j|$ for $j \in \{1,2\}$. As $\{u,v\} \subset Z$ is a good pair and both these vertices are either of type~1 or~2, we have
			\begin{align*}
				cn &\le |\Gamma_{X}(u)\triangle \Gamma_{X}(v)|\\
				&=|\Gamma_{S}(u)\triangle \Gamma_{S}(v)|+|\Gamma_{T}(u)\triangle \Gamma_{T}(v)|+|\Gamma_{X'}(u)\triangle \Gamma_{X'}(v)|\\
				&<cn/4+|\Gamma_{T}(u)\triangle \Gamma_{T}(v)|+cn/2\\
				&=3cn/4+b_1+b_2,
			\end{align*}
			so either $b_1$ or $b_2$ is at least $cn/8$. Also, both $|\Gamma_{T}(u)|$ and $|\Gamma_{T}(v)|$ are at most $m-cn/8$ as $u$ and $v$ are both non-problematic vertices (and hence not of type~4). Therefore, both $b_1$ and $b_2$ are at most $m-cn/8$. Consequently, at least one of $b_1$ or $b_2$ lies between $cn/8$ and $m-cn/8$. The claim now follows in this case as an easy consequence of Proposition~\ref{bineq}. 
			
			The case where $u$ and $v$ are both either of type~3 or~4 follows analogously by estimating $\max_{r\in \mathbb{Z}}\{\PV(\alpha_{2}-\alpha_{1}=r)\}$ instead; this may be done using the same argument as above, but with the roles of $S$ and $T$ interchanged.
			
			Finally, the case where $u$ and $v$ are both of type~0 may be addressed as follows. It follows from the definition of a type~0 vertex that all of $|A_1|/m$, $|A_2|/m$, $|B_1|/m$ and $|B_2|/m$ are bounded away from $1$, and furthermore, one of these quantities is bounded away from $0$ since $\{u,v\}$ is a good pair; the result then follows by bounding one of $\max_{r\in \mathbb{Z}}\{\PV(\alpha_{2}-\alpha_{1}=r)\}$ or $ \max_{r\in \mathbb{Z}}\{\PV(\beta_{2}-\beta_{1}=r)\}$, as appropriate.
		\end{proof}
		
		We are now ready to prove Claim~\ref{important}.
		\begin{proof}[Proof of Claim~\ref{important}]
			As we remarked earlier, we have
			\[ \sum_{(k,i)\in\C{I}}\EV[e(D_{k,i})]=\sum_{(k,i)\in\C{I}}\sum_{\{u,v\}\subset Z}\PV(\wdeg_{k,i}(u)=\wdeg_{k,i}(v)).\]
			To bound this sum, we decompose it into three parts. First, let $\C{S}_1$ denote the family of all pairs $((k,i),\{u,v\})$ with $(k,i) \in \C{I}$ and $\{u,v\} \subset Z$ such that $\{u,v\}$ is good and $(k,i)$ is compatible with $\{u, v\}$. Next, let $\C{S}_2$ denote the family of all pairs $((k,i),\{u,v\})$ with $(k,i) \in \C{I}$ and $\{u,v\} \subset Z$ such that $\{u,v\}$ is good and $(k,i)$ is incompatible with $\{u, v\}$. Finally, let $\C{S}_3$ denote the family of all pairs $((k,i),\{u,v\})$ with $(k,i) \in \C{I}$ and $\{u,v\} \subset Z$ such that $\{u,v\}$ is bad. We may then write
			\[\sum_{(k,i)\in\C{I}}\EV[e(D_{k,i})]= S_1 + S_2 +S_3, \]
			where
			\[S_{j}=\sum_{\C{S}_j}\PV(\wdeg_{k,i}(u)=\wdeg_{k,i}(v))\]
			for $1 \le j \le 3$.
			
			First, by Claim~\ref{subclaim1}, each term in the sum defining  $S_{1}$ is at most $4\exp(-n^{\delta})$. As this sum consists of at most $n^{4}$ terms, we have $S_{1}<4n^{4}\exp(-n^{\delta})<1$.
			
			Next, by Claim~\ref{totalbadpairs}, the sum defining $S_{2}$ consists of at most $10n^{7/2+\delta}\log n$ terms and by Claim~\ref{subclaim2}, each of these terms is at most $c'/n^{1/2}$. Consequently, we have $S_{2}<10c'n^{3+\delta}\log n$.
			
			Finally, recall that by Lemma~\ref{bipartitediverse}, the number of bad pairs $\{u,v\} \subset Z$ is at most $n^{1+\delta}$, so the sum defining $S_{3}$ consists of at most $n^{3+\delta}$ terms. Therefore, $S_{3}<n^{3+\delta}$.
			
			Putting everything together, we see that
			\[\sum_{(k,i)\in\C{I}}\EV[e(D_{k,i})] \le  1+10c'n^{3+\delta}\log n+ n^{3+\delta} < n^{3+2\delta}. \qedhere \]
		\end{proof}
		
		To finish the proof, note that by Claim~\ref{important}, we have
		\[\sum_{(k,i)\in\C{I}}e(D_{k,i})<n^{3+2\delta}\]
		with positive probability. By Claim~\ref{lowerbound},
		\[\Psi(G,\omega)\ge\sum_{(k,i)\in\C{I}}N(k,i),\]
		so by~(\ref{equ1}), with positive probability, we have
		\[\sum_{(k,i)\in\C{I}}N(k,i)>10^{-8}c^{6}n^{6-2\delta}\left(n^3 + n^{3+2\delta}\right)^{-1}>n^{3-\eps}\]
		provided $\eps > 4\delta$, proving Theorem~\ref{secondstep}.
	\end{proof}
	
	We are now in a position to prove Theorem~\ref{mainthm}, our main result.
	
	\begin{proof}[Proof of Theorem~\ref{mainthm}]
		Without loss of generality, suppose that $0 < \eps<1$ and fix $\eps_{0} = \eps/ 100$. All inequalities in the sequel will hold provided $n$ is sufficiently large.
		
		By Theorem~\ref{firststep}, there exist constants $c_{1}, c_2>0$ depending on $C$ and $\eps$ alone such that the following holds for all sufficiently large $n \in \N$. Writing $s = c_{1}n^{1/2}/5$, for each integer $1 \le i \le s$, there exist disjoint sets $U_{i}\subset V$ and $W'_{i}\subset V$ such that
		\begin{enumerate}
			\item $c_{1}n^{2}+(5i-2)n^{3/2} \le e(G[U_{i}]) \le c_{1}n^{2}+(5i+2)n^{3/2}$,
			\item $|W'_{i}| \ge n^{1/2-\eps_{0}}$,
			\item there exists a positive real number $l_{i} \ge c_{2}n$ such that $l_{i} \le \Deg_{U_{i}}(x) \le l_{i}+n^{1/2+\eps_{0}}$ for all $x\in W'_{i}$, and
			\item the degrees $\Deg_{U_{i}}(x)$ are distinct for all $x\in W'_{i}$.
		\end{enumerate}
		
		For each $1 \le i \le s$, choose a subset $W_{i} \subset W'_i$ of size exactly $c_{2}n^{1/2-\eps_{0}}/2$. To prove the result, we shall only consider subgraphs induced by sets of the form $U_{i}\cup Z$, where $Z\subset W_{i}$. The following two facts will prove useful.
		
		\begin{enumerate}[(A)]
			\item\label{propA} If $i < j$, then $e(G[U_{i}\cup Z_1])<e(G[U_{j}\cup Z_2])$ for all $Z_1 \subset W_{i}$ and $Z_2\subset W_{j}$. To see this, note that
			\begin{align*}
				e(G[U_{i}\cup Z_1]) & \le e(G[U_{i}])+n|Z_1| < c_{1}n^{2}+(5i+2)n^{3/2}+n^{3/2}\\
				&\le e(G[U_{j}])\le e(G[U_{j}\cup Z_2]).
			\end{align*}
			
			\item\label{propB} If $Z_1,Z_2\subset W_{i}$ are subsets satisfying $|Z_1|<|Z_2|$, then $e(G[U_{i}\cup Z_1])<e(G[U_{i}\cup Z_2])$. To see this, note that
			\[e(G[U_{i}\cup Z_1]) \le e(G[U_{i}])+(l_{i}+n^{1/2+\eps_{0}})|Z_1|+|Z_1|^{2}\]
			and that
			\[e(G[U_{i}\cup Z_2])\ge e(G[U_{i}])+l_{i}|Z_2|;\]
			it follows that
			\begin{align*}
				e(G[U_{i}\cup Z_2])-e(G[U_{i}\cup Z_1]) &\ge l_{i}-|Z_1|n^{1/2+\eps_{0}}-|Z_1|^{2}\\
				&\ge c_{2}n-c_{2}n/2-n^{1-2\eps_{0}}>0.
			\end{align*}
		\end{enumerate}
		
		Since $G$ is $C$-Ramsey, the largest clique or independent set in any induced subgraph of $G$ has size at most $C\log n$. As $|W_{i}|=c_{2}n^{1/2-\eps_{0}}/2$, this implies that $G[W_{i}]$ is $C_{0}$-Ramsey for each $1 \le i \le s$, where $C_{0}=2C/(1-3\eps_{0})$. Define a weight function $\omega_{i}: W_{i}\rightarrow \mathbb{N}\cup \{0\}$ by $\omega_{i}(v)=\Deg_{U_{i}}(v)-l_{i}$. It is clear that
		\[e(G[U_{i}\cup Z])=e(G[U_{i}])+l_{i}|Z| +e^{\omega_{i}}(G[Z]),\]
		so, by~\eqref{propB}, the set $\{e(G[U_{i}\cup Z]):Z\subset W_{i}\}$ has exactly $|\Psi(G[W_{i}],\omega_{i})|$ elements.
		Clearly, $\omega_{i}$ is injective; furthermore, we also have
		\[\omega_{i}(v)\le n^{1/2+\eps_{0}}\le|W_{i}|^{1+10\eps_{0}}\] for all $v \in W_i$. By applying Theorem~\ref{secondstep} to $G[W_{i}]$ (with parameters $C=C_{0}$, $\omega=\omega_{i}$, $\delta=10\eps_{0}$ and $\eps=50\eps_{0}$), we see that
		\[|\Psi(G[W_{i}],\omega_{i})| \ge |W_{i}|^{3-50\eps_{0}} \ge(c_{2}/2)^{3}n^{3/2-60\eps_{0}}.\]
		
		The result now follows from~\eqref{propA} because
		\[\Phi(G)\ge \sum_{i=1}^{s} |\Psi(G[W_{i}],\omega_{i})| \ge c_{1}(c_{2}/2)^{3}n^{2-60\eps_{0}}/5 \ge n^{2-\eps}. \qedhere\]
	\end{proof}
	\section{Conclusion}\label{conc}
	Many of the currently known properties of Ramsey graphs have been obtained by first showing that Ramsey graphs satisfy certain `quasirandomness' conditions and then demonstrating the property in question for all graphs satisfying those conditions. In this spirit, we believe that an analogue of Theorem~\ref{mainthm} should hold for all graphs whose edges are reasonably well-distributed. An $n$-vertex graph is said to be \emph{uniformly $\eps$-dense} if the edge density of any induced subgraph on at least $n^\eps$ vertices lies between $\eps$ and $1-\eps$. It is easily seen from Theorem~\ref{density} that Ramsey graphs are uniformly dense. We conjecture the following strengthening of Theorem~\ref{mainthm}
	
	\begin{conjecture}
		For any fixed $\eps >0$, if $G$ is a uniformly $\eps$-dense graph on $n$ vertices, then $|\Phi(G)| = n^{2-o(1)}$.
	\end{conjecture}
	
	We have shown that $|\Phi(G)| = n^{2-o(1)}$ for any Ramsey graph $G$ on $n$ vertices. Using effective versions of Theorems~\ref{density} and~\ref{diversity}, it is in fact possible to read out a lower bound of the form $n^2/\omega(n)$ from our proof, where  $\omega(n) = \exp(\Theta((\log n)^{1/2}))$ is a slowly growing error term. Let us mention, for the sake of the reader interested in the specifics of the aforementioned estimate, that the main bottleneck in our argument is Theorem~\ref{diversity}; improved bounds for this result should immediately translate into better estimates for the error term in the main result proved here. We naturally believe that our result should hold with a suitable positive constant in the place of the error term $\omega$, and it remains an interesting open problem to prove such a statement; we suspect a better understanding of the large-scale structure of Ramsey graphs will be required to settle this question.
	\bibliographystyle{amsplain}
	\bibliography{ramsey_sizes}
	
\end{document}